\documentclass[12pt]{article}
\usepackage[latin9]{inputenc}
\usepackage{amsthm}
\usepackage{amsmath}
\usepackage{amssymb}
\usepackage{graphicx}

\makeatletter
\usepackage{enumitem}		
\theoremstyle{plain}
\newtheorem{thm}{\protect\theoremname}
\theoremstyle{plain}
\newtheorem{cor}[thm]{\protect\corollaryname}
\theoremstyle{definition}
\newtheorem{example}[thm]{\protect\examplename}
\theoremstyle{definition}
\newtheorem{defn}[thm]{\protect\definitionname}
\theoremstyle{plain}
\newtheorem{lem}[thm]{\protect\lemmaname}
\theoremstyle{plain}
\newtheorem{prop}[thm]{\protect\propositionname}
\theoremstyle{plain}
\newtheorem{conj}[thm]{\protect\conjecturename}

\theoremstyle{definition}
\newtheorem{claim}{\protect\claimname}
\newtheorem*{claim*}{\protect\claimname}

\usepackage{color}\usepackage{fullpage}\usepackage{epsf}\usepackage{amsfonts}\usepackage{amsthm}\usepackage{color}\usepackage{transparent}

\DeclareMathOperator{\E}{\mathbb{E}}

\usepackage{pdfsync}

\makeatother

\providecommand{\corollaryname}{Corollary}
\providecommand{\definitionname}{Definition}
\providecommand{\examplename}{Example}
\providecommand{\lemmaname}{Lemma}
\providecommand{\propositionname}{Proposition}
\providecommand{\theoremname}{Theorem}
\providecommand{\claimname}{Claim}
\providecommand{\conjecturename}{Conjecture}

\begin{document}
\title{On directed versions of the Corrádi-Hajnal Corollary}

\author{Andrzej Czygrinow\thanks{School of Mathematical Sciences and Statistics, Arizona
State University, Tempe, AZ 85287, USA. E-mail address: aczygri@asu.edu.}
\and
H. A. Kierstead
\thanks{School of Mathematical Sciences and Statistics, Arizona
State University, Tempe, AZ 85287, USA. E-mail address: kierstead@asu.edu.
Research of this author is supported in part by NSA grant H98230-12-1-0212.}
\and
Theodore Molla\thanks{School of Mathematical Sciences and Statistics, Arizona
State University, Tempe, AZ 85287, USA. E-mail address: tmolla@asu.edu.
Research of this author is supported in part by NSA grant H98230-12-1-0212.}
}

\maketitle
\begin{abstract}
For $k\in \mathbb N$, Corr\'adi and Hajnal proved that every graph $G$ on $3k$ vertices with minimum degree $\delta(G)\ge2k$ has a $C_3$-factor, i.e., a partitioning of the vertex set so that each part induces the $3$-cycle $C_3$. Wang proved that  every directed graph $\overrightarrow G$ on $3k$ vertices with minimum total degree $\delta_t(\overrightarrow G):=\min_{v\in V}(deg^-(v)+deg^+(v))\ge   3(3k-1)/2$
 has a $\overrightarrow C_3$-factor, where $\overrightarrow C_3$ is the directed $3$-cycle.  The degree bound in Wang's result is tight. However,
our main result implies that for all integers $a\ge1$ and $b\ge0$ with $a+b=k$, every directed graph $\overrightarrow G$ on $3k$ vertices with minimum total degree $\delta_t(\overrightarrow G)\ge   4k-1$ has a factor consisting of $a$  copies of $\overrightarrow T_3$ and $b$ copies of 
 $ \overrightarrow C_3$, where $\overrightarrow T_3$ is the transitive tournament on three vertices. In particular, using $b=0$, there is a $\overrightarrow T_3$-factor of $\overrightarrow G $, and using $a=1$, it is possible to obtain a $\overrightarrow C_3$-factor of $\overrightarrow G$ by reversing just one edge of $\overrightarrow G$. All these results are phrased and proved more generally in terms of undirected multigraphs. 

We conjecture that every directed graph $\overrightarrow G$ on $3k$ vertices with minimum semidegree $\delta_0(\overrightarrow G):=\min_{v\in V}\min(deg^-(v),deg^+(v))\ge2k$ has a $\overrightarrow C_3$-factor, and prove that this is asymptotically correct.
\end{abstract}

\section{Introduction}

For a graph $G=(V,E)$ set $|G|:=|V|$ and $\|G\|:=|E|$. Let $d(v)$ denote the degree of a vertex $v$, $\delta(G):=\min\{d(v):v\in V\}$ denote the \emph{minimum degree} of $G$, and $\sigma_{2}:=\min_{vw\notin E(G)}d(v)+d(w)$ denote the \emph{minimum Ore-degree} of $G$.  
Two subgraphs of $G$ are \emph{independent} if their vertex sets are disjoint.
In 1963 Corrádi and Hajnal \cite{corradi1963maximal} proved: 
\begin{thm}
\label{CHT}Every graph $G$ with $|G|\ge3k$ and 
$\delta(G)\geq2k$ contains $k$ independent cycles.
\end{thm}
 In 1998, Enomoto \cite{E} 
proved an Ore-type version of Theorem~\ref{CHT}:
\begin{thm}
\label{EnoTh}Every graph $G$ with $|G|\geq3k$ and
$\sigma_2(G)\geq4k-1$ contains $k$ independent cycles. 
\end{thm}
A \emph{tiling }of a graph $G$ is a set of independent subgraphs, called \emph{tiles}.
A tiling is a \emph{factor}, if its union spans $G$. For a subgraph $H\subseteq{G}$,
an ${H}$-tiling is a tiling whose tiles are all
isomorphic to ${H}$.  The $3$-cycle $C_3$ is called a triangle.
Theorem~\ref{CHT} has the following corollary, whose complementary
version is a precursor and special case of the 1970 Hajnal-Szemerédi
\cite{hajnal1970pcp} theorem on equitable coloring. 
\begin{cor}
\label{CHC}Every graph $G$ with $|G|$ divisible by $3$ and 
$\delta(G)\geq\frac{2|G|}{3}$ has a $C_3$-factor. 
\end{cor}
This paper is motivated by the problem of proving versions of Corollary~\ref{CHC}
for directed graphs $\overrightarrow{G}:=(V,\overrightarrow{E}$).
Our directed graphs are \emph{simple}, in the sense that they have 
no loops and for all $x,y\in V$ there are at most two edges, one
of form $\overrightarrow{xy}$ and one of form $\overrightarrow{yx}$,
whose ends are in $\{x,y\}$. The \emph{in}- and \emph{out-degree} of a vertex $v$ are denoted by $deg^-(v)$ and $deg^+(v)$; the  \emph{total degree} of 
$v$ is the sum $deg_t(v):=deg^{-}(v)+deg^{+}(v)$, and the \emph{semidegree} of $v$ is $deg_0(v):=\min\{deg^-(v),deg^+(v)\}$.
The \emph{minimum total degree} of $\overrightarrow{G}$ is 
$\delta_t(\overrightarrow{G}):=\min\{deg_t(v):v \in V\}$, 
 and  the \emph{minimum semidegree} of $\overrightarrow{G}$ is 
 $\delta_0(\overrightarrow{G}):=\min\{deg_0(v):v\in V\}$.

Let $\overrightarrow{K}\subseteq\overrightarrow{G}$ be a subgraph
on three vertices $x,y,z$. If $\{\overrightarrow{xy},\overrightarrow{yz},\overrightarrow{xz}\}\subseteq E(\overrightarrow{K})$
then $\overrightarrow{K}$ is a transitive triangle, denoted by $\overrightarrow{T}_{3}$;
if $\{\overrightarrow{xy},\overrightarrow{yz},\overrightarrow{zx}\}\subseteq E(\overrightarrow{K})$
then $\overrightarrow{K}$ is a \emph{cyclic} triangle, denoted by
$\overrightarrow{C}_{3}$.  

Wang \cite{wangdir} proved the following directed version of Corollary
\ref{CHC}.
\begin{thm}
\label{thm:Wang}Every directed graph $\overrightarrow{G}$  with $\delta_t(\overrightarrow{G)}\ge\frac{3|\overrightarrow{G}|-3}{2}$ has a $\overrightarrow{C}_3$-tiling of size 
$\lfloor\frac{\overrightarrow{|G|}}{3}\rfloor$.
\end{thm}
The degree condition of Theorem~\ref{thm:Wang} is tight: 

\begin{example}\label{ex:Wang}
Suppose
$n=2k+1$ is odd and divisible by $3.$ Let $\overrightarrow G$ be the directed graph
with $V(\overrightarrow G)=V_{1}\cup V_{2}$, where $V_{1}\cap V_{2}=\emptyset$,
$k=|V_{1}|$, $|V_{2}|=k+1$, and 
\[
E(\overrightarrow G)=F:=\{\overrightarrow{xy}:x,y\in V_{1}\vee x,y\in V_{2}\vee(x\in V_{1}\wedge y\in V_{2})\}.
\]
Then $\delta_t(\overrightarrow G)=3k-1=\frac{3|\overrightarrow G|-3}{2}-1$, and no $\overrightarrow{C}_3$
contains vertices from both parts $V_{1}$ and $V_{2}$. So no tiling
of $\overrightarrow G$ contains $\frac{n}{3}$ cyclic triangles. 
\end{example}

While Theorem~4 is tight, we can significantly relax the minimum
degree condition for the cost of at most one incorrectly oriented
edge. Our main result implies:
\begin{cor}[to Theorem \ref{th:main}]
\label{cor:main}Suppose $\overrightarrow{G}$
is a directed graph with $\delta_t(\overrightarrow{G})\geq\frac{4|\overrightarrow G|-3}{3}$,
and $c\geq0$ and $t\geq1$ are integers with $c+t=\lfloor\frac{|\overrightarrow G|}{3}\rfloor$.
Then $\overrightarrow{G}$ has a tiling consisting of $c$ cyclic
triangles and $t$ transitive triangles. 
\end{cor}
The degree condition of Corollary~\ref{cor:main} is also tight. 

\begin{example}
There exists a directed graph $\overrightarrow{G}$ such that $\delta_t(\overrightarrow{G})\geq\frac{4|\overrightarrow G|-3}{3}-1$,
but $\overrightarrow{G}$ cannot be tiled with any combination of
cyclic and transitive triangles: Suppose $|\overrightarrow G|=3k$.
Set $V(\overrightarrow{G})=V_{1}\cup V_{2}$, where $V_{1}\cap V_{2}=\emptyset$,
$|V_{1}|=k+1$, and $|V_{2}|=2k-1$, and let $E(\overrightarrow{G})=\{\overrightarrow{xy}:x\notin V_{1}\wedge y\notin V_{1}\}$.
Then $\delta_t(\overrightarrow{G})=4k-2=\frac{4n-3}{3}-1$,
but $\overrightarrow{G}$ does not have any triangle factor, since
every vertex in $V_{1}$ would need to be paired with two vertices
from $V_{2}$, and there are too few vertices in $V_{2}$.
\end{example}

Our methods are obscured by the elementary proofs of Theorem~\ref{th:main} and Corollary~\ref{cor:main}. We used stability techniques (regularity-blow-up, absorbing structures) to discover what should be true, and only then were able to concentrate our energy on a successful elementary argument of the optimal result. Sometimes the process works in the other direction.  The striking gap between Wang's Theorem and Corollary~\ref{cor:main}, suggests that there is a better theorem. An important characteristic of Example~\ref{ex:Wang} is that while every vertex has large total degree, it also has small semidegree. 
   This led us to the following conjecture.

\begin{conj}\label{con1}
Every directed graph $\overrightarrow G$ with $|\overrightarrow G|=3k$ and $\delta_0(\overrightarrow G)\ge2k$ has a $\overrightarrow C_3$-factor. 
\end{conj}

In support of this conjecture, we use stability techniques to prove the following asymptotic version.  We expect that with more effort this approach can be improved to a proof of the conjecture for sufficiently large graphs by using the techniques similar to those of Levitt, S\'ark\"ozy and Semer\'edi  \cite{levitt2010avoid}.

\begin{cor}[to Theorem \ref{thm:5tri_factor}]
\label{cor:stability}For every $\varepsilon>0$
there exists  $n_{0}\in\mathbb N$ such that every directed graph $\overrightarrow{G}$
with $|\overrightarrow G| \ge n_0$ and minimum semidegree 
 $\delta_0(\overrightarrow{G})\geq(\frac{2}{3}+\varepsilon)|G|$
has a $\overrightarrow C_3$-tiling of size $\left\lfloor\frac{|\overrightarrow G|}{3}\right\rfloor$.
\end{cor}

It turns out that our results can be phased more generally, and
proved more easily,  in
terms of multigraphs. Suppose that $M=(V,E)$ is a multigraph. For
two vertices $x,y\in V$ let $\mu(x,y)$ be the number of edges with
ends $x$ and $y$.  In particular, if $xy\notin E$ then $\mu(x,y)=0$.  Let $\mu(M):=\max_{x,y\in V}\mu(x,y)$. 
 The \emph{degree} of $x$ is $d(x):=\sum_{y\in V}\mu(x,y)$. 
The \emph{minimum degree} of $M$ is $\delta(M):=\min\{d(v):v\in V\}$. A
$k$-\emph{triangle} is a multigraph $T_k$ such that $C_3\subseteq T_k$,
$|T_k|=3$ and $\|T_k\|=k$.

The \emph{underlying multigraph} $M$ of a simple directed graph $\overrightarrow{G}$
is obtained by removing the orientation of all edges of $G$. In particular,
if $\overrightarrow{xy}$ and $\overrightarrow{yx}$ are both edges of
$\overrightarrow{G}$ then $\mu_{M}(x,y)=2$. By our definitions,
$\delta_t(\overrightarrow{G})=\delta(M)$. If $M$ contains a $4$-triangle with vertices $x,y,z$ and $\mu_{M}(x,y)=2$ then $\overrightarrow{G}$
contains a transitive triangle with the same vertex set, since regardless
of the orientation of $xz$ and $yz$, one of the orientations $\overrightarrow{xy}$
or $\overrightarrow{yx}$ completes a transitive triangle on $\{x,y,z\}$.
Thus if $M$ contains $t$ independent $4$-triangles  then $\overrightarrow{G}$
contains $t$ independent transitive triangles. Notice that the converse is not true. For instance, if $\overrightarrow{G}$ is an orientation then $\mu(M)=1$, and so $M$ contains no $4$-triangle. Similarly, if $M$ contains a $5$-triangle with vertices $x,y,z$ and $\mu(xy)=2=\mu(yz)$ then  $\overrightarrow{G}$ contains a cyclic triangle (and also a transitive triangle) on the same vertex set, since any orientation of the edge $xz$ can be extended to a cyclic triangle by choosing the orientations of the other two edges carefully.


It is convenient to introduce the following terminology and notation. A multigraph $M:=(V,E)$ is \emph{standard} if  $\mu(M)\leq2$. For a fixed standard multigraph $M$ we use the following default notation. Two simple graphs
$G:=G_{M}:=(V,E_{G})$ and $H:=H_{M}:=(V,E_{H})$ are defined by $E_{G}:=\{xy:\mu(x,y)\geq1\}$
and $E_{H}:=\{xy:\mu(x,y)=2\}$. Edges $xy\in E_{H}$ are said to
be \emph{heavy}, and $y$ is said to be a \emph{heavy} neighbor of
$x$; edges in $E_{G}\smallsetminus E_{H}$ are \emph{light}. We also set $n:=|M|$.

The following is our
main theorem.
\begin{thm}
\label{th:main}
Every standard multigraph
$M$ with  $\delta(M)\geq\frac{4n-3}{3}$
has a tiling  of size $\left\lfloor\frac{n}{3}\right\rfloor$, where one tile is a $4$-triangle and the remaining $\left\lfloor\frac{n}{3}\right\rfloor-1$ tiles are   $5$-triangles.
\end{thm}

Other authors have studied different degree conditions for directed graphs. 
For example Ghouila-Houri \cite{gh} and Woodall \cite{wood}  proved analogs of Dirac's and Ore's theorems for directed graphs. Orientations of graphs (directed graphs with neither multiple edges nor $2$-cycles) lead to another group of results. For example,  Keevash, Kuhn, and Osthus \cite{kko} proved that if $\overrightarrow{G}$ is an oriented graph then $\delta_0(\overrightarrow{G})\geq (3|\overrightarrow{G}| - 4)/8$ guarantees the existence of a Hamilton cycle. Recently, Keevash and Sudakov \cite{ks} showed that
there is some
   $\epsilon>0$ such that for  sufficiently large $n$
    if $\overrightarrow{G}$ is an oriented graph on at least $n$ vertices with $\delta_0(\overrightarrow{G})\geq (1/2 - \epsilon)|\overrightarrow{G}|$ then $\overrightarrow{G}$ contains a packing of directed triangles covering all but at most three vertices.  

The paper is organized as follows. In the remainder of this section we review some additional notation.  In Section 2 we warm up by giving a short self-contained proof of the generalization of Wang's Theorem to standard multigraphs.  This generalization is needed in Section 4.  In Section 3 we prove our main result, and in Section 4 we prove stability results related to Conjecture~\ref{con1}. 

\subsection{Additional notation}

Fix a multigraph (or graph) $M=(V,E)$.
 Set $|M|:=|V|$ and $\left\Vert M\right\Vert :=|E|$.
For a subset $U\subseteq V$, let $\left\Vert U\right\Vert :=\left\Vert M[U]\right\Vert =\frac{1}{2}\sum_{e\in E(U)}\mu(e)$; and
 let $\overline{U}:=V\smallsetminus U$. 
 For any vertex $v\in V(M)$, let $\left\Vert v,U\right\Vert :=\sum_{u\in U}\mu(v,u)$. 
For $U,U'\subseteq V(G)$, let $\left\Vert U,U'\right\Vert :=\sum_{v\in U}\left\Vert v,U'\right\Vert $,
and let $E(U,U')$ be the set of edges with one end in $U$ and the
other end in $U'$. Then $\left\Vert U,U'\right\Vert =|E(U,U')|+\left\Vert U\cap U'\right\Vert $. Viewing an edge or nonedge $xy$ as a set,  these definitions imply $\mu(xy)=\|xy\|$.
If $T$ is a triangle
and $V(T)=\{x,y,z\}$ we may identify $T$ by listing the vertices
as $xyz$ or by listing an edge and a vertex, that is, if $e=yz$
then we may refer to $T$ as $ex$ or $xe$.

Set $[n]:=\{1,\dots,n\}$. If $i,j\in [n]$ is clear from the context, we may write $i\oplus j$ for $i+j\pmod n$ with out explicitly mentioning $n$.

\section{Warm-up}\label{section:4-triangles}

In this section we warm up by proving the  multigraph generalizations of Theorem~\ref{thm:Wang} and the case $c=0$ of Corollary~\ref{cor:main}. For completeness, and to illustrate the origins of our methods, 
we begin with a short proof of Corollary~\ref{CHC} based on Enomoto's proof of Theorem~\ref{EnoTh}.

\begin{proof}[Proof of Corollary~\ref{CHC}]
Let $G=(V,E)$ be an edge-maximal counterexample. 
 Then $n=3k$, $\delta(G)\geq2k$,
 $G$ does not contain a $C_3$-factor (so $G\ne K_{3k}$),
but the graph $G^{+}$ obtained by adding a new edge $a_{1}a_{3}$
does have a $C_3 $-factor. So $G$ has a \emph{near triangle factor} $\mathcal{T}$, 
i.e., a factor such that $A:=a_{1}a_{2}a_{3}\in\mathcal{T}$
is a path and every $H\in\mathcal{T}-A$ is a triangle. 
\begin{claim*}
Suppose $\mathcal{T}$ is a near triangle factor of $G$ with
path $A:=a_{1}a_{2}a_{3}$ and triangle $B:=b_{1}b_{2}b_{3}$. If
$\left\Vert \{a_{1},a_{3}\},B\right\Vert \ge5$ then $\left\Vert a_{2},B\right\Vert =0$.\end{claim*}
\begin{proof}
Choose notation so that $\left\Vert a_{1},B\right\Vert =3$ and $\left\Vert a_{3},B\right\Vert \geq2$.
Suppose $b_{i}\in N(a_{2})$. Then either $\{a_{1}b_{i\oplus1}b_{i\oplus2},a_{2}a_{3}b_{i}\}$
or $\{a_{1}a_{2}b_{i},a_{3}b_{i\oplus1}b_{i\oplus2}\}$ is a $C_3$-factor of $G[A\cup B]$, depending on whether $b_{i}\in N(a_{3})$.
Regardless, this contradicts the minimality of $G$. 
\end{proof}
Since $\left\Vert \{a_{1},a_{3}\},G\right\Vert \geq4k$, but $\left\Vert \{a_{1},a_{3}\},A\right\Vert =2<4$,
there is a triangle $B:=b_{1}b_{2}b_{3}\in\mathcal{T}$ with $\left\Vert \{a_{1},a_{3}\},B\right\Vert \geq5$.
 Choose notation so that $b_1 ,b_2,b_3\in N(a_1)$ and  $b_2,b_3\in  N(a_3)$. 
 Applying  the claim to $A$ yields $\left\Vert a_{2},B\right\Vert =0$. 
Thus 
\[
2\left\Vert \{b_{1},a_{2}\},A\cup B\right\Vert +\left\Vert \{a_{1},a_{3}\},A\cup B\right\Vert \leq2(4+2)+2(1+3)=20<24=6\cdot2\cdot2.
\]
 Since $2\left\Vert \{b_{1},a_{2}\},G\right\Vert +\left\Vert \{a_{1},a_{3}\},G\right\Vert \geq12k$,
some triangle $C:=c_{1}c_{2}c_{3}\in\mathcal{T}$ satisfies:
\[
2\left\Vert \{b_{1},a_{2}\},C\right\Vert +\left\Vert \{a_{1},a_{3}\},C\right\Vert \geq13.
\]
Then $\left\Vert a_{2},C\right\Vert,\left\Vert \{a_{1},a_3\},C\right\Vert >0$. By Claim, $\left\Vert \{a_{1},a_{3}\},C\right\Vert \leq4$;
 so $\left\Vert \{b_{1},a_{2}\},C\right\Vert \geq5$. Claim applied
 to $\mathcal{T}\cup\{b_{1}a_{1}a_{2},a_{3}b_{2}b_{3}\}\smallsetminus\{A,B\}$ yields
 $\left\Vert a_{1},C\right\Vert =0$. 
So $\left\Vert a_{3},C\right\Vert >0$ and either 
$\left\Vert \{b_{1},a_{2}\},C\right\Vert =6$ 
or  $\left\Vert a_3,C\right\Vert = 3$. Thus some $i\in [3]$
satisfies $c_ia_2,c_ia_3,c_{i\oplus1}b_1,c_{i\oplus2}b_1\in E(G)$. So
$\mathcal T\cup \{c_ia_2a_3,b_1c_{i\oplus1}c_{i\oplus2},a_1b_2b_3\}\smallsetminus \{A,B,C\}$ is a  $C_3$-factor of  $G$.
\end{proof}

Next we use Corollary~\ref{CHC} to prove Theorem~\ref{thm:4-triangles}.
\begin{thm}\label{thm:4-triangles}
  Every standard multigraph $M$
  on $n$ vertices with $\delta(M)\geq\frac{4n-3}{3}$
  contains $\lfloor\frac{n}{3}\rfloor$ independent $4$-triangles.
\end{thm}
\begin{proof}
  We consider three cases depending on $n \pmod 3$.

  \emph{Case 0}: $n \equiv 0\pmod 3$. 
   Since 
  $\delta(G_M)\ge \lceil\frac{1}{2}\delta(M)\rceil \ge \frac{2}{3}n$, Corollary~\ref{CHC}
  implies $M$ has a  triangle factor $\mathcal{T}$.
  Choose $\mathcal T$ 
  having the maximum number of $4$-triangles.
  We are done, unless $\|A\| = 3$ for some  $A=a_1a_2a_3 \in \mathcal{T}$. 
  Since $\|A, M\| \ge 3 \left(\frac{4n-3}{3}\right)$,
  \begin{equation*}
    \|A, M-A\| \ge 4n - 3 - \|A, A\| = 4n - 9 > 
    12 \left( \frac{n-3}{3} \right).
  \end{equation*}
  Thus $\|A, B\| \ge 13$ for some $B=b_1b_2b_3 \in \mathcal{T}$.
  Suppose $\|a_1, B\| \ge \|a_2, B\| \ge \|a_3, B\|$. 
  Then  $5 \le \|a_1, B\| \le 6$ and  
   $\|\{a_2, a_3\}, B\| \ge 7$.
  Hence,  
  $\|\{a_2, a_3\}, b_i\| \ge 3$ for some $i\in[3]$; so $\mathcal T\cup\{a_2a_3b_i,a_1b_{i\oplus1}b_{i\oplus2}\}\smallsetminus\{A,B\}$ is a $4$-triangle factor of $M$.

  \emph{Case 1}: $n\equiv 1\pmod 3$. 
Pick $v\in V$, and set $M':=M-v$. Then $|M'|\equiv 0\pmod 3 $, and $$\delta(M')\ge\delta(M)-\mu(M)\ge \left\lceil\frac{4n-9}{3}\right\rceil=\frac{4(n-1)-3}{3} \ge \frac{4|M'|-3}{3}.$$ By Case 0, $M'$, and also $M$, contains $\lfloor\frac{|M'|}{3}\rfloor=\lfloor\frac {n}{3}\rfloor$ independent $4$-triangles.

  \emph{Case 2}: $n\equiv 2\pmod 3$. Form $M^+\supseteq M$ by adding a new
  vertex $x$ and heavy edges  $xv$ for all $v\in V(M)$. Then $|M^+|\equiv 0\pmod3$ and $\delta(M^+)\ge  \frac{4|M^+|-3}{3}$. By Case 0, $M^+$ contains $\frac{|M^+|}{3}$ independent $4$-triangles. So $ M=M^+-x$ contains $\frac{|M^+|}{3}-1=\lfloor\frac{n}{3}\rfloor$ of them. 
\end{proof}

Now we consider $5$-triangle tilings. First we prove Proposition~\ref{prop:one_vertex_4}, which is also needed in the next section. Then we strengthen Wang's Theorem to standard multigraphs.

\begin{prop}
  Let $T=v_1v_2v_3 \subseteq M$ be a $5$-triangle, and $x \in V(M - T)$.
 If $3\le\left\Vert x, T\right\Vert \le 4$
  then $xe$ is a $(\left\Vert x, T\right\Vert+1)$-triangle for some $e\in E(T)$.
  \label{prop:one_vertex_4}
\end{prop}
\begin{proof}
  Suppose 
  $v_1v_2,v_1v_3\in E_H$. If $N(x)\subseteq \{v_2v_3\}$ then $xv_2v_3$ is a $(\left\Vert x, T\right\Vert+1)$-triangle.
  Else,  $\|x,v_1v_i\|\ge\|x,T\|-1$ for some  $i \in \{2, 3\}$. 
  So $xv_1v_i$ is a $(\|x,T\|+1)$-triangle.
\end{proof}



\begin{thm}\label{thm:stWang} 
  Every standard multigraph $M$ with 
  $\delta(M) \ge \frac{3n - 3}{2}$ contains $\lfloor\frac{n}{3}\rfloor$ independent $5$-triangles.
\end{thm}
\begin{proof}
  Consider two cases depending on whether $n\equiv 2\pmod 3$. 

  \emph{Case 1}:
  $n\not\equiv2\pmod 3$. 
  By Theorem~\ref{thm:4-triangles}, $M$ has a tiling 
  $\mathcal{T}$ consisting of $\lfloor\frac{n}{3}\rfloor$ independent $4$- and $5$-triangles.  
  Over all such tilings, 
  select $\mathcal{T}$ with the maximum number of $5$-triangles.
  We are done, unless there exists $A=a_1a_2a_3 \in \mathcal{T}$ such
  that $\left\Vert A \right\Vert = 4$.
  Assume $a_1a_2$ is the heavy edge of $A$.  
  By the case, $L:=V\smallsetminus\bigcup\mathcal T$ has at most one vertex.
  If $L\ne\emptyset$ then let $a'_3\in L$; otherwise set $a'_3:=a_3$. Also set
  $A':=A+a'_3$. Then 
  $\|a'_3,A\|\le2+2(|A'|-3)=2|A'|-4$, since otherwise $G[A']$ contains a $5$-triangle. So $\|A,A'\smallsetminus A\|\le4(|A'|-3)$.

  For  $B \in \mathcal{T}$, define 
   $ f(B) := \|A,B\|+ \|a_3', B\|$.
  Then $f(A) = 8+\|a'_3,A\|\le4+2|A'|$. 
  So
  \begin{align*}  
    \sum_{B \in \mathcal{T} 
  } f(B)
  &=d(a_1) + d(a_2) + d(a_3)+d(a'_3) - \|A,A'\smallsetminus A\|
  \ge 4\cdot \frac{3n - 3}{2}
  -\|A,A'\smallsetminus A\|\\
  &\ge 6n-6 -4(|A'|-3) = 6(n-|A'|)+(4+2|A'|)+2\\
  &> 18\left(|\mathcal{T}| - 1\right) + f(A).
\end{align*}
Thus $f(B) \ge 19$ for some $B\in \mathcal T-A$. If $B$ is
a $4$-triangle then set $B':=B+e'$, where $e'$ is parallel to some $e\in E(B)$
with $\mu(e)=1$, and set $M':=M+e'$. Otherwise, set $B':=B$ and $M':=M$. It suffices to prove that $M'[A'\cup B']$ contains two independent $5$-triangles, since in either case another $5$-triangle can be added to  $\mathcal T$, a contradiction.

Label the vertices of $B'$
as $b_1, b_2, b_3$ so that $b_1b_2$ and $b_1b_3$ are heavy
edges. 
Since $a_1a_2$ is a heavy edge, 
if $\left\Vert \{a_1, a_2\}, b\right\Vert \ge 3$ then $a_1a_2b$ is a $5$-triangle for all $b\in V(B)$.
Consider three cases based
on $k:=\max\{\|a_3,B\|,\|a'_3,B\|\}$. Let $a\in\{a_3,a'_3\}$ satisfy $\|a,B\|=k$. 
Since $f(B) \ge 19$, we have $4 \le \left\Vert a, B\right\Vert \le 6$.  

If $\left\Vert a, B\right\Vert = 4$ then 
$\left\Vert \{a_1, a_2\}, B\right\Vert \ge 11$. By Proposition~\ref{prop:one_vertex_4},  there exists 
$i \in [3]$ such that $ab_ib_{i\oplus1}$ is a $5$-triangle,  and
$a_1a_2b_{i\oplus2}$ is another disjoint $5$-triangle.

If $\left\Vert a, B\right\Vert = 5$ then 
$\left\Vert \{a_1, a_2\}, B\right\Vert \ge 9$. So there exists $ i \in \{2,3\}$ such that 
$a_1a_2b_i$ is a $5$-triangle; and $ab_1b_{5-i}$ is another $5$-triangle.

Finally, if $\left\Vert a, B\right\Vert = 6$ then 
$\left\Vert \{a_1, a_2\}, B\right\Vert \ge 7$. So there exists $i \in [3]$ such that $a_1a_2b_i$ is
a $5$-triangle; and
 $ab_{i\oplus 1}b_{i\oplus2}$ is another $5$-triangle.

\emph{Case 2}: $n\equiv 2\pmod 3$. Form $M'\supseteq M$ by adding a vertex $x$
and heavy edges $xv$ for all $v\in V$. Then $|M'|\equiv 0\pmod3$ and $\delta(M')\ge  \frac{3|M'|-3}{2}$. By Case 1, $M'$ contains $\frac{|M^+|}{3}$ independent $5$-triangles. So $ M=M'-x$ contains $\frac{|M'|}{3}-1=\frac{n}{3}$ of them. 
\end{proof}

\section {Main Theorem}
In this section we prove our main result, Theorem~\ref{th:main}. Let $M$ be a
standard multigraph with $\delta(M)\ge\frac{4n - 3}{3}$. We start with  three Propositions used in the proof.
\begin{prop} 
  Suppose $T=v_1v_2v_3 \subseteq M$ is a $5$-triangle, and $x_1,x_2 \in V(M-T)$ are distinct vertices with
  $\left\Vert \{x_1, x_2\}, T \right\Vert \ge 9$.
  Then $M[\{x_{1}, x_{2}\} \cup V(T)]$ has a factor containing
  a $5$-triangle and an edge $e$ such that $e$ is heavy if $\min_{i\in[2]}\{\|x_i,T\|\}\ge4$. 
  \label{prop:two_vertices_9} 
\end{prop}
\begin{proof}
  Label so that $v_1v_2,v_1v_3\in E_H$ and
  $\left\Vert x_1, T \right\Vert \ge \left\Vert x_2, T \right\Vert$.
  
  First suppose  $\|x_2,T\|\ge4$. If $x_2v_i\in E_H$ for some $i\in\{2,3\}$ then $\{x_1v_1v_{5-i},x_2v_i\}$ works. Else  $V(T)\subseteq N(x_2)$. Also $x_1v_j\in E_H$ for some $j\in\{2,3\}$. So $\{x_1v_j,x_2v_1v_{5-j}\}$ works.

Otherwise, $\left\Vert x_2, T \right\Vert = 3$
  and $\left\Vert x_1, T \right\Vert = 6$.
So  $\{x_1v_{i\oplus1}v_{i\oplus2},x_2v_i\}$ works for some $i\in[3]$.
\end{proof}

\begin{prop}
  Suppose $T=v_1v_2v_3 \subseteq M$ is a $5$-triangle, and $e_{1},e_{2} \in E(M-T)$
  are independent heavy edges with
  $\left\Vert e_{1}, T \right\Vert \ge 9$ and
  $\left\Vert e_{2}, T \right\Vert \ge 7$.
  Then $M[e_{1}\cup e_{2}\cup V(T)]$ contains 
  two independent $5$-triangles.
  \label{prop:two_heavy_edges_9_7} 
\end{prop}
\begin{proof}
  Choose notation  so that
  $\left\Vert e_1, v_i \right\Vert \ge 3$ for both $i \in [2]$.
  There exists $j\in [3]$ so that
  $\left\Vert e_2, v_j \right\Vert \ge 3$.
  Pick $i \in [2]-j$. Then
  $e_1v_i$ and $e_2v_j$ are disjoint $5$-triangles. 
\end{proof}

\begin{prop}
Suppose $T \subseteq M$ is a $5$-triangle, and 
  $xyz$ is a path in $H_M-T$.
  If $\|xz, T\| \ge 9$ and $\|y, T\| \ge 1$ then
  $M[\{x, y, z\} \cup V(T)]$ has a factor containing a $5$- and a $4$-triangle.
  \label{prop:heavy_path_5_4_triangles} 
\end{prop}
\begin{proof}
  Choose notation so that $\|x, T\| \ge \|z, T\|$, and $T=v_1v_2v_3$ with 
  $v_1\in N(y)$. We identify a $4$-triangle $A$ and a $5$-triangle $B$ depending on several cases.

  Suppose $\|x, T\| = 6$ and $\|z, T\| \ge 3$.
  If $zv_1\in E$ then  set $A:=yzv_1$ and $B:=xv_2v_3$; else 
  set $A:=zv_2v_3$ and $B:= xyv_1$. Otherwise $\|x, T\| = 5$ and $\|z, T\| \ge 4$.
  
  If $zv_1\notin E$ then set
  $A:=xyv_1$ and $B:=zv_2v_3$. Otherwise  $zv_1\in E$.
  
  If $zv_1$ is  heavy then set $A:=xv_2v_3$ and $B:=zyv_1$; if $xv_1$ is  light 
  then set $A:=zyv_1$ and $B:=xv_2v_3$. Otherwise $zv_1$ is  light  and $xv_1$ is  heavy. Set
  $A:=zv_2v_3$ and $B:=xyv_1$. 
\end{proof}
\begin{proof}[Proof of Theorem \ref{th:main}]
  We consider three cases depending on $n\pmod 3$.

  \bigskip\noindent
  \emph{Case 0}: $n\equiv 0\pmod 3$. Let $n=:3k$, and let $M$ be  a maximal counterexample. Let $\mathcal{T}$ be a maximum
$T_{5}$-tiling of $M$ and $U=\bigcup_{T\in\mathcal{T}}V(T)$. 
\begin{claim}\label{clm:1}
$|\mathcal{T}|=k-1$.\end{claim}
\begin{proof} 
Let $e\in\overline{E}$. By the maximality of $M$, $M+e$ has a factor
$\mathcal{T}'$ consisting of 5-triangles and one $4$-triangle
$A_{1}$. If $e\in A_{1}$ then the $5$-triangles are contained in
$M$, and so we are done. Otherwise, $e\in E(A_{2}^{+})$ for some
$5$-triangle $A_{2}^{+}\in\mathcal{T}'$. Set $A_{2}:=A_{2}^{+}-e$,
and $A:=A_{1}\cup A_{2}$. Then $A$ satisfies: (i) $|A| = 6$, (ii) $M[A]$ contains two independent
heavy edges, and (iii)  $M[V \smallsetminus A]$ has a $T_5$-factor.
Over all vertex sets satisfying (i--iii), 
select $A$ and independent
heavy edges $e_1,e_2\in M[A]$ so that $\|z_1z_2\|$ is maximized,
where $\{z_1, z_2\} := A \smallsetminus (e_1 \cup e_2)$.   
Let $\mathcal{T}'$ be a $T_5$-factor of $M[V \smallsetminus A]$. Set $A_i:=e_i+z_i$, for $i\in[2]$.

If $M[A]$ contains a $5$-triangle we are done. Otherwise
$\left\Vert x,A_{2}\right\Vert \leq4$ for all $x\in V(A_{1})$, and
so $\left\Vert A\right\Vert =\left\Vert A_{1}\right\Vert +\left\Vert A_{2}\right\Vert +\left\Vert A_{1},A_{2}\right\Vert \leq20$.
Thus 
\[
\left\Vert A,V\smallsetminus A\right\Vert \geq6\left(\frac{4}{3}n-1\right)-40>24(k-2).
\]
So $\left\Vert A,B\right\Vert \geq25$ for some $B = b_1b_2b_3 \in\mathcal{T}'$.
It suffices to show that $M[A\cup B]$ contains  two independent $T_{5}$.

Suppose $\left\Vert \{z_{1},z_{2}\},B\right\Vert \geq9$. If there exists $h
\in [2]$ such that $\left\Vert z_{h},B\right\Vert =6$
then choose $i\in[3]$ with $\left\Vert b_{i},A\right\Vert \geq9$.
There exists $j\in[2]$ with $\left\Vert b_{i}e_{j}\right\Vert \geq5$;
also $\left\Vert z_{h}b_{i\oplus1}b_{i\oplus2}\right\Vert \geq5$.
So we are done. Otherwise, $\left\Vert z_{h},B\right\Vert \geq4$
and $\left\Vert z_{3-h},B\right\Vert \geq5$ for some $h\in[2]$.
By Proposition~\ref{prop:two_vertices_9}, $M[V(B)+z_{1}+z_{2}]$ has a factor consisting of
a heavy edge and a $T_{5}$, implying, by the maximality
of $\|z_1z_2\|$, that $\left\Vert z_{1}z_{2}\right\Vert =2$.
Set $e_{3}:=z_{1}z_{2}$. 
Choose distinct $i,j \in [3]$ so that 
$\|e_i, B\| \ge 9$ and $\|e_{j}, B\| \ge 7$. 
By Proposition~\ref{prop:two_heavy_edges_9_7},
there are two $T_5$ in $M[e_i \cup e_{j} \cup V(B)]$. 
\end{proof}

By Claim \ref{clm:1}, $W: =V\smallsetminus U$ satisfies $|W|=3$. Choose
$\mathcal T$ with $\|W\|_H$ maximum.
  
  \begin{claim}\label{clm:6}
 $\|W\|\ge4$. 
  \end{claim}
  \begin{proof}
   Suppose not. Then $\|W,U\|\ge 3(\frac{4}{3}n-1)-3>12(k-1)$. So $\|W,T\|\ge13$ for some $T\in\mathcal T$.
    Thus there exist $w,w'\in W$ with $\|w,T\|\ge4$ and $\|w',T\|\ge5$. By
    Proposition~\ref{prop:two_vertices_9}, $M[V(T)\cup\{w,w'\}]$ has a factor containing a $T_5$ and a heavy edge. By the choice of $W$ this implies $\|W\|_H=1$. Set $W=:\{x,y,z\}$ where $xy$ is heavy. 
   Since
    \begin{equation*}
      2 \|z, U\| + \|xy, U\| \ge 4\left(\frac{4n}{3} - 1\right) - 2 - 5
      > 16 \left(\frac{n}{3} - 1 \right) 
      = 16(k-1),
    \end{equation*}
    some $T=v_1v_2v_3 \in \mathcal{T}$
    satisfies $2 \|z, T\| + \|xy, T\| \ge 17$. Suppose $v_1v_2,v_1v_3\in E_H$. To contradict the maximality of $\|W\|$, it suffices to find $i\in[3]$ so that $\{M[\{x,y,v_i\}],M[\{z,v_{i\oplus1},v_{i\oplus2}\}]\}$ contains a $T_5$, and a graph with at least four edges.

    If $\|z, T\| = 3$ then $\|xy, T\| \ge 11$. Choose $i\in\{2,3\}$ so that $\|z,v_1v_{i}\|\le1$.

    If $\|z, T\| = 4$ then $\|xy, T\| \ge 9$.
    Choose $i\in\{2,3\}$ so that $xyv_i$ is a $5$-triangle.

    If $\|z, T\| = 5$ then $\|xy, T\| \ge 7$.
    Choose $i \in \{2,3\}$ so that $\|xy, v_i\| \ge 2$.

    Otherwise, $\|z, T\| = 6$ and $\|xy, T\| \ge 5$. Choose $i \in [3]$ so that $\|xy, v_i\| \ge 2$.
  \end{proof}
 
Since $M$ is a counterexample and $\|W\|\ge4$, we have $M[W]=:xyz$ is a  path in $M_H$. 
  \begin{claim}\label{clm:7}
    There exists $A \in \mathcal{T}$ and a labeling 
    $\{a_1, a_2, a_3\}$ of $V(A)$ such that
    \begin{enumerate} [label=(\alph*),nolistsep, ref={\theclaim~(\alph*)}]
      \item
	\label{clm:7a}
	$x$ is adjacent to $a_1$;
      \item
	\label{clm:7b}
	one of $xa_2a_3$ and $za_2a_3$ is a $5$-triangle
	and the other is at least a $4$-triangle; 
      \item
	\label{clm:7c}
	if $xa_1$ is light then 
	both $xa_2a_3$ and $za_2a_3$ are $5$-triangles; and
      \item
	\label{clm:7d}
	$\|y, A\| = 0$.
    \end{enumerate}
  \end{claim}
  \begin{proof}
    There 
    exists $A=a_1a_2a_3 \in \mathcal{T}$ such that $\|xz, A\| \ge 9$, since
    $$\|xz, U\| \ge2\left(\frac{4n}{3} - 1\right)- \|xz, W\|\ge \frac{8n}{3} -
    2-4 > 8\left(\frac{n}{3} - 1 \right) = 8(k-1).$$
    Say $\|x, A\| \ge \|z, A\|$. Since $M$ is a
    counterexample, Proposition
    \ref{prop:heavy_path_5_4_triangles}  implies (d) $\|y, A\| = 0$.

    If $\|z, A\| = 3$ then, by Proposition~\ref{prop:one_vertex_4}, $za_2a_3$ is a $4$-triangle for some
     $a_2,a_3 \in A$.
    In this case $\|x, A\| = 6$, so
    $xa_2a_3$ is a $5$-triangle and $xa_1$ is a heavy edge. So (a--c) hold.
    
    If $\|z, A\| \ge 4$ then, by Proposition~\ref{prop:one_vertex_4}, 
    $za_2a_3$ is a $5$-triangle for some $a_2a_3 \in A$.
    In this case $\|x, A\| \ge 5$, so 
    $xa_2a_3$ is a $4$-triangle and $x$ is adjacent to $a_1$.
    Furthermore, if $xa_1$ is light then $xa_2a_3$ is a $5$-triangle. 
    Again (a--c) hold.
  \end{proof}
  
  \begin{claim}\label{clm:8}
    There exists $B \in \mathcal{T} - A$ such that
    $2\|a_1y, B\| + \|xz, B\| \ge 25$.
  \end{claim}
  \begin{proof} Set $U':=U \smallsetminus V(A)$. 
    Since $xz\notin E$ and $\|y,A\|=0$,

    \begin{align*}
      2\|a_1y, U'\| + \|xz, U'\| &\ge
      6\left(\frac{4}{3}n - 1\right)-
      2\|a_1y, W \cup V(A)\| - \|xz, W \cup A\|\\ 
      &\ge \frac{24n}{3}-6-2(8+4)-(8+8) 
      > 24 \left(\frac{n}{3} - 2\right) = 24(k-2).
    \end{align*}
   So there exists $B \in \mathcal{T} - A$ with $2\|a_1y, B\| + \|xz, B\| \ge 25$.
  \end{proof}

  Let $W' := W \cup \{a_1\} \cup V(B)$.
  For any edge $e \in \{a_1x, xy, yz\}$ define
  $$Q(e) := \{u \in V(B): \|e, u\| \ge 3\}$$ and
  for any vertex $v \in \{a_1, x, y, z\}$ and $k \in \{4, 5\}$ define
  \begin{equation*}
    P_k(v) := \{u \in B: 
      T_k \subseteq M[B - u + v] \}. 
  \end{equation*}
  \begin{claim}\label{clm:9}
  If $v\notin e$  
  and there exists $u\in P_k(v) \cap Q(e)$ then $M[\left(V(B) \cup e\right) + v]$
  can be factored into a $(3+\|e\|)$-triangle and a $k$-triangle.
  Moreover:
  \begin{equation}\label{PQ}
    \begin{split}
    \text{(a) }&|Q(e)| \ge  \frac{\|e, B\| - 6}{2} ~~ 
    \text{(b) }|P_5(v)| \ge \|v, B\| - 3~~ \\
    \text{(c) }&|P_4(v)| = 3 \text{ if $\|v, B\| \ge 5$ and }
               |P_4(v)| \ge (\|v,B\| - 2)  \text{ otherwise.}
    \end{split}
  \end{equation}
  \end{claim}
  \begin{proof}
For the first sentence apply definitions; for \eqref{PQ} check each  argument value.
\end{proof}
To obtain a contradiction, it suffices to
  find two independent triangles $C,D\subseteq M[W'-w]$ for some $w \in \{a_1, x, y\}$
so that $\{C,D,wa_2a_3\}$ 
  is a factor of $M[W \cup V(A) \cup V(B)]$
  consisting of two  $5$-triangles and
  one  $4$-triangle. We further refine this notation by setting
  $D:=vb_{i\oplus1}b_{i\oplus2}$ and $C:=b_ie$, where 
  $v \in \{a_1, x, y, z\}$,
  $b_i\in B$ and $e\in
  E(\{a_1, x, y, z\} - v)$. Then $w$ is defined by 
  $w\in \left(\{a_1,x,y,z\}\smallsetminus e\right)-v$; set $W^*:=wa_2a_3$. 

\begin{claim}\label{clm:10}
None of the following statements is true:
\begin{enumerate} [label=(\alph*),nolistsep, ref={\theclaim~(\alph*)}]
    \item
      \label{clm:10a}
      $P_4(v) \cap Q(e) \neq \emptyset$
      for some  $e \in \{xy, yz\}$ and
      $v \in \{x, z\}\smallsetminus e$. 
    \item
      \label{clm:10b}
      $P_5(v) \cap Q(e) \neq \emptyset$ for some 
      $e \in \{a_1x, xy, yz\}$ and $v \in \{a_1,x, y, z\}\smallsetminus e$ 
      such that $y \in e +v$.
    \item
      \label{clm:10c}
      $P_4(a_1) \cap Q(xy) \neq \emptyset$ and 
      $P_4(a_1) \cap Q(yz) \neq \emptyset$.
    \item
      \label{clm:10d}
   There exists $b_i\in P_5(a_1)$ such that $x,y,z\in N(b_i)$.
  \end{enumerate}
  \end{claim}
  \begin{proof}
  By Claim~\ref{clm:7}, each case implies $M$ is not a counterexample, a contradiction:
  
  (a) Then $w=a_1$, and so $\|W^*\|\ge5$, $\|C\|\ge5$ and $\|D\|\ge4$.
  
  (b) Then $\|D\|\ge5$. If $\|e\|=2$ then $\|C\|\ge5$ and $\|W^*\|\ge4$;
  otherwise $e = a_1x$ and, by Claim~\ref{clm:7c}, $\|W^*\|\ge5$ and $\|C\|\ge4$.
  
  (c) By Claim~\ref{clm:7b}
  , $\|wa_2a_3\|\ge5$ for some $w\in \{x,z\}$. Set $v:=a_1$ and $e:=\{x,y,z\}-w$. Then $\|W^*\|\ge5$, $\|C\|\ge5$ and $\|D\|\ge4$.
  
  (d) Set $v:=a_1$, choose $w\in \{x,z\}$ so that $\|W^*\|\ge5$, and set $e:=\{x,y,z\}-w$. Then $\|D\|\ge5$ and $\|C\|\ge4$.
 \end{proof}
   \begin{claim}
     \label{clm:11}
    $\|a_1, B\| < 5$. 
  \end{claim}
  \begin{proof}
    Suppose not. 
     Let $\{x',z'\}=\{x,z\}$, where $\|x',B\|\ge\|z',B\|$. 
    For $k \in \{4, 5\}$, define 
    \begin{equation*} 
      s_k(e,v):=|Q(e)| + |P_k(v)|.
    \end{equation*} 
  Then $s_k(e,v)>3$ implies $Q(e)\cap P_k(v)\ne\emptyset$. We use Claim~\ref{clm:9} to calculate $s_k(e,v)$. Observe 
  $$25-2\|a_1,B\|\le \|x'y,B\|+\|yz',B\|,\text{ and so }\|x'y,B\|\ge13-\|a_1,B\|.$$
  
  If $\|a_1,B\|=6$ then  $s_5(x'y,a_1)\ge 1 + 3$, contradicting
  Claim~\ref{clm:10b}. Otherwise, $\|a_1,B\|=5$. Either $\|x'y,B\|\ge9$ or
  $\|z'y,B\|\ge7$. In the first case, $s_5(x'y,a_1)\ge2+2$, contradicting
  Claim~\ref{clm:10b}. In the second case, $s_4(x'y,a_1), s_4(z'y,a_1)>1+3$, contradicting Claim~\ref{clm:10c}. 
  \end{proof}

  \begin{claim}\label{clm:12}
    $\|\{a_1,y\}, B\| <9$.
  \end{claim}
  \begin{proof}
Suppose $\|\{a_1, y\},B\| \ge9$. We consider several cases. 


\emph{Case 1}: $\|a_1,B\|=4$ and $\|y,B\|=6$. By Proposition~\ref{prop:one_vertex_4} there are distinct  $b,b',b''\in V(B)$ with $b\in P_5(a_1)$ and $1\le \|a_1,b'\|\le\|a_1,b''\|=2$. Claim~\ref{clm:10b} implies $b\notin Q(xy)\cup Q(yz)$; so $x,z\notin N(b)$, since $\|b,y\|=2$. 
By Claim~\ref{clm:9}~(b),  
 $P_5(y)=B$; so $Q(a_1,x) 
 =\emptyset$ by Claim~\ref{clm:10b}. Thus $\|x,b'\|\le 2-\|a_1,b'\|$ and $\|x,b''\|=0$. 
By the case $\|\{x,z\},B\|\ge5$; thus $\|x,b'\|=1$ and $\|z,\{b',b''\}\|=4$; so $\|a_1,b'\|=1=\|a_1,b\|$. Thus $b'\in P_4(a_1)\cap Q(xy)\cap Q(yz)$, contradicting Claim~\ref{clm:10c}.

%

\emph{Case 2}: $\|a_1,B\|=3$ and $\|y,B\|=6$. Then $\|\{x,z\},B\|\ge7$. For $\{u,v\}=\{x,y\}$, we have
$\|u,B\|\ge 1$. So Claim~\ref{clm:9}~(a) implies $|Q(uy)|\ge1$. Thus Claim~\ref{clm:10a} implies $|P_4(v)|\le 2$. So Claim~\ref{clm:9} (c) implies $\|v,B\|\le4$. Thus $3\le\|x,B\|,\|z,B\|\le4$.

By Proposition~\ref{prop:one_vertex_4}, $b\in P_4(a_1)$ for some $b\in B$. So $b\notin Q(xy)\cap Q(yz)$ by Proposition~\ref{clm:10b}. Thus $b\notin N(x)\cap N(z)$. 
 Also $P_5(y)=B$ by $\|y,B\|=6$.  By Claim~\ref{clm:10b}, $Q(a_1x)=\emptyset$. Thus  $\|x,B-b\|\le2$. So $xb\in E$ and $\|z,B-b\|=\|z,B\|\ge3$. Thus $b\in P_4(z)\cap Q(xy)$, contradicting Claim~\ref{clm:10a}.


\emph{Case 3}:
$\left\Vert a_{1},B\right\Vert =4$ and $\left\Vert y,B\right\Vert =5$.
Then (i) $\left\Vert \{x,z\},B\right\Vert \geq7$; let $\{x',z'\}:=\{x,z\}$,
where $\left\Vert x',B\right\Vert \geq\left\Vert z',B\right\Vert $.
Claim~\ref{clm:9} (b) implies $|P_{5}(y)|\geq2$; Proposition~\ref{prop:one_vertex_4}
implies $b_{i}\in P_{5}(a_{1})$ for some $i\in[3]$. So by Claim~\ref{clm:10} (b,d),
(ii) $|Q(a_{1}x)|\leq1$, (iii) $b_{i}\notin Q(xy)\cup Q(yz)$, and
(iv) $xyz\nsubseteq N(b_{i})$. Thus by (iii) and Claim~\ref{clm:9} (a), 
$\left\Vert x',B\right\Vert \leq5$,
and so by (i), $\left\Vert z',B\right\Vert \geq2$. 
By Claim~\ref{clm:9} (a),
$|Q(x'y)|\geq2$ and $|Q(yz')|\geq1$. 
Claim~\ref{clm:10a}  then implies $|P_4(x')| \le 2$ 
meaning
(v) $4 \ge \left\Vert x',B\right\Vert \ge \left\Vert z',B\right\Vert \ge 3$
and, therefore, by Claim~\ref{clm:9} (a), (vi) $|Q(a_{1}x)|\geq1$. 
By Claim~\ref{clm:10} (b,c) $|P_{5}(a_{1})|\leq1$
and $|P_{4}(a_{1})|\leq2$. 
This implies (vii) $b_{i}\notin N(a_{1})$: Otherwise,
since $b_i \in P_5(a_1)$ implies $\|b_i, a_1\| \le 1$, 
there exist $h,j\in[3]-i$ with $\left\Vert b_{h},a_1\right\Vert =2$ and
$\left\Vert b_{i},a_1\right\Vert =1=\left\Vert b_j, a_{1}\right\Vert $. 
If $\left\Vert b_{i}b_{j}\right\Vert =1$ then $|P_{5}(a_{1})|=2$;
if $\left\Vert b_{i}b_{j}\right\Vert =2$ then $|P_{4}(a_{1})|=3$.
Either is a contradiction.

By (ii), (vi) and (vii), $Q(a_{1}x)=\{b_{j}\}$ for some $j\in[3]-i$,
and $\left\Vert a_{1},b_{h}b_{j}\right\Vert =4$, where $h=6-i-j$.
Thus $N(x)\cap B=\{b_{i},b_{j}\}$. By (iv), $z\notin N(b_{i})$,
and so $N(z)\cap B=\{b_{h},b_{j}\}$. So
\begin{align*}
\left\Vert yzb_{h}\right\Vert  & =\left\Vert z,B\right\Vert +\left\Vert y,zb_{h}\right\Vert -\left\Vert zb_{j}\right\Vert \geq3+3-2=4 & \mbox{ by (v)}\\
\left\Vert xb_{i}b_{j}\right\Vert  & =\left\Vert x,B\right\Vert +\left\Vert b_{i}b_{j}\right\Vert \geq3+1=4 & \mbox{by (v)}\\
\left\Vert yzb_{h}\right\Vert +\left\Vert xb_{i}b_{j}\right\Vert  & = 
\|\{x,z\}, B\| + \|y, zb_h\| - \|zb_j\| + \|b_ib_j\| \ge \left\Vert \{x,z\},B\right\Vert +2\geq9 & \mbox{by (i)}
\end{align*}
Thus $\{yzb_{h},xb_{i}b_{j},A\}$ is a factor of $M(A\cup B\cup W)$
with two $T_{5}$ and a $T_{4}$, a contradiction.
  \end{proof}
  
  \begin{claim}\label{clm:13}
    $\|\{a_{1},y\}, B\| \ge9$. 
  \end{claim}
  \begin{proof}
%
    Suppose $\|\{a_{1},y\}, B\| \le8$. Then
     $\|\{x,z\}, B\| \ge 9$ and $\|y, B\| \ge 1$.  
    Proposition \ref{prop:heavy_path_5_4_triangles} implies   
    there exist independent $4$- and $5$-triangles
    in $M[W \cup V(B)]$, a contradiction.
  \end{proof}
Observing that Claim~\ref{clm:12} contradicts Claim~\ref{clm:13}, completes the proof of Case~0.

  \bigskip
  \noindent\emph{Case 1}: $n\equiv1\pmod3$.  Choose any vertex $v\in V$, and
  set $M'=M-v$. By Case 0, $M'$, and so $M$, contains
  $\frac{n-4}{3}=\lfloor\frac{n-3}{3}\rfloor$ independent $5$-triangles and
  a $4$-triangle.

  \bigskip
  \noindent\emph{Case 2}: $n\equiv2\pmod3$. Add a new vertex $x$ together with all edges of the form $xv,v\in V$ to $M$ to get $M^+$. By Case 0, $M^+$ contains $\frac{n-2}{3}= \lfloor\frac{n-3}{3}\rfloor+1$ independent $5$-triangles
  and a $4$-triangle, at most one of them contains $x$. So $M$ contains
  $\lfloor\frac{n-3}{3}\rfloor$ independent $5$-triangles and a $4$-triangle.
\end{proof}

\section{Relaxed degree conditions \label{section:cyclic_triangles} }
\global\long\def\ab{\frac{1}{32}}
In this section we attack Conjecture~\ref{con1} and prove Corollary~\ref{cor:stability}. We rely on ideas from Levitt, S\'ark\"ozy and Semer\'edi  \cite{levitt2010avoid}.
\begin{example}
  Let $M$ be a underlying multigraph of the directed graph $\overrightarrow G$ from Example~\ref{ex:Wang}. Then $|M|=n=2k + 1$,  $V=V_1\cup V_2$, $V_1\cap V_2=\emptyset$, $|V_1|=k$, $|V_2|=k+1$,
  every pair $xy$  is an edge, and $xy$ is heavy if and only if $x,y\in V_1\vee x,y\in V_2$.  
  Since no $5$-triangle contains vertices of both $V_1$ and $V_2$, and
   $|V_1|$ is not divisible by $3$,
  $M$ does not have a $5$-triangle factor.
  Moreover,
  $ \delta(M) = 2(|V_1| - 1) + |V_2| = 3k - 1 = \frac{3n - 3}{2} - 1.$ Finally, notice that $E_H(V_1,V_2)=\emptyset$.
  \label{ex:5tri_counter}
\end{example}

Example~\ref{ex:5tri_counter} shows that Theorem~\ref{thm:stWang}
is tight but it also suggests that
requiring $H_M$ to be connected may allow us to
relax the degree condition.

\begin{conj}
  If $M$ is a standard multigraph on $n$ vertices
  where $n$ is divisible by $3$,
  $\delta(M) \ge \frac{4}{3}n - 1$ and $H_M$ is connected
  then $M$ has a $T_5$-factor.
  \label{conj:5tri}
\end{conj}
As a side note, Conjecture~\ref{conj:5tri} implies both both Corollary~\ref{CHC} 
 and the following
Theorem of Enomoto, Kaneko and Tuza \cite{enomoto87}.
\begin{thm}
  If $G$ is a connected graph on $3k$ vertices and
  $\delta(G) \ge k$ then $G$ has a $k$ independent paths on $3$ vertices.
\end{thm}

\begin{defn}
  For $\alpha \ge 0$ call a graph $G$ on $n$ vertices
  $\alpha$-splittable if there exists a partition
  $\{A, B\}$ of $V(G)$ such that 
  $|A|,|B| \ge \frac{n}{3}$ and
  $\left\Vert A, B\right\Vert_{G} \le \alpha n^2$.
\end{defn}

In this section, we will prove the following theorem, 
which supports Conjecture~\ref{conj:5tri}, and yields Corollary~\ref{cor:stability}.
\begin{thm}
  \label{thm:5tri_factor}
  For every $\varepsilon, \alpha > 0$ there exists
  $n_0 := n_0(\varepsilon, \alpha)$ such that 
  for every standard multigraph $M = (V, E)$ on $n \ge n_0$ vertices
  where $n$ is divisible by $3$ the following holds.
  If $\delta(M) \ge \left(\frac{4}{3} + \varepsilon\right)n$
  and $H_M$ is not $\alpha$-splittable
  then $M$ has a $5$-triangle factor.
\end{thm}

Before the proof, 
we will first collect a few simple facts and definitions that will
be used throughout, and show that Theorem~\ref{thm:5tri_factor} implies Corollary~\ref{cor:stability}.

Let $\varepsilon > 0$ and 
let $M = (V, E)$ be standard multigraph on $n$ vertices
such that $\delta(M) \ge \left(\frac{4}{3} + \varepsilon\right)n$.
Note that for every $u \in V$
\begin{equation*}
  |N(u)| \ge \left(\frac{2}{3} + \frac{\varepsilon}{2}\right)n
  \text{ and }
  |N_H(u)| \ge \left(\frac{1}{3} + \varepsilon\right)n.
  \label{eq:relax_degree_cond}
\end{equation*}

For any $U \subseteq V$ and $k \ge 1$ define
$Q_k(U) := \{ v \in V: \|v, U\| \ge k\}$. 
For any $e \in E$,
\begin{equation*}
  2\left(\frac{4}{3} + \varepsilon\right)n \le \|e, V\| 
  \le |Q_4(e)| + 3|Q_3(e)| + 2|\overline{Q_3(e)}|
  = |Q_4(e)| + |Q_3(e)| + 2n.
\end{equation*}
Therefore, 
\begin{equation}
  |Q_4(e)| + |Q_3(e)| \ge \left(\frac{2}{3} + 2\varepsilon\right)n,
  \label{eq:q_4_plus_q_3}
\end{equation}
and since $Q_4(e) \subseteq Q_3(e)$,
\begin{equation}
  |Q_3(e)| \ge \left(\frac{1}{3} + \varepsilon\right)n.
  \label{eq:q_3}
\end{equation}

For any $u \in V$, let 
$F(u) := \{e \in E_H : u \in Q_3(e) \}$.
Note that 
  $$2|F(u)| \ge \sum_{v \in N_H(u)} |N(u) \cap N_H(v)|,$$ 
and for every $v \in N_H(u)$,
$|N(u) \cap N_H(v)| \ge 
\left(\frac{2}{3} + \frac{\varepsilon}{2} \right)n + 
\left(\frac{1}{3} + \varepsilon\right)n - n
= \frac{3 \varepsilon}{2}n$.
So 
\begin{equation}
  |F(u)| \ge 
  \frac{1}{2}
  \left(\frac{1}{3} + \varepsilon\right) \frac{3\varepsilon}{2}n^2
  > \frac{\varepsilon}{4} n^2.
  \label{eq:f_u}
\end{equation}


We are now ready to prove Corollary~\ref{cor:stability}. 
 At the same time we will prove the following corollary showing that Conjecture~\ref{conj:5tri} implies Conjecture~\ref{con1}.
 
 \begin{cor} \label{cor:CiC}
 If every standard multigraph $M$ 
  with  $|M|=3k$, 
  $\delta(M) \ge 4k - 1$, and $H_M$ connected
   has a $T_5$-factor, then every directed graph $\overrightarrow G$ with $|\overrightarrow G|=3k$ and $\delta_0(\overrightarrow G)\ge2k$ has a $\overrightarrow C_3$-factor. 
 \end{cor}

\begin{proof}[Proof of Corollaries~\ref{cor:stability} and \ref{cor:CiC}]
For Corollary~\ref{cor:stability} we are given $\varepsilon$ 
and 
 set $\alpha := \frac{\varepsilon}{4}$; 
 for Corollary~\ref{cor:CiC} set $\alpha, \varepsilon:=0$.
 Let
  $M$ be the underlying multigraph of $\overrightarrow G$.
  Note that $\delta(M) \ge \left(\frac{4}{3} + 2\varepsilon \right)n$, and
  $\delta(H_M)\ge (\frac{1}{3}+2\varepsilon)n$.
  If $H := H_M$ is not $\alpha^2$-splittable then
  Corollary~\ref{cor:stability} follows from Theorem~\ref{thm:5tri_factor}.
  Moreover, if $\alpha=0$ then $H_M$ is connected---not $0$-splittable implies 
  connected---and so 
  Conjecture~\ref{conj:5tri} implies Conjecture~\ref{con1}.
  So assume $H$ is $\alpha^2$-splittable, but
   $\overrightarrow G$ does not have a $\overrightarrow C_3$-factor. 


  Partition 
  $V := V(H)$ as $\{A, B\}$
  so that
  $|A|, |B| \ge \frac{n}{3}$ and $e_H(A, B) \le \alpha^2 n^2$.
  Set 
  \begin{equation*}
    Z := \{v \in V : |E_H(v) \cap E_H(A, B)| > 2 \alpha n\}
  \end{equation*}
  and note that $|Z| \le \alpha n$.
  For every $z \in Z$, $|N_H(z)| \ge \left(\frac{1}{3} + 8\alpha\right)n$ 
  so we can find a matching $K$ in $H$ such
  that $Z \subseteq \bigcup{K}$.
  By \eqref{eq:q_3} each $e \in K$ satisfies 
    $Q_3(e) \ge \left(\frac{1}{3} + \varepsilon\right)n$; 
   so we can select $x_e \in (V \smallsetminus \bigcup{K}) \cap Q_3(e)$ 
   so that $Y := \bigcup_{e \in K} ex_e$ is a collection
  of disjoint $5$-triangles. 
  Let
  $V' := V \smallsetminus Y$, 
  $A' := A \smallsetminus Y$ and $B' := B \smallsetminus Y$.
  Then  $|Y| \le 3 |Z| \le 3\alpha n$ and $|E_H(v) \cap E_H(A', B')| \le 2
  \alpha n$ for every $v \in V'$. So 
  \begin{align}\label{V'}
  \text{(i) } \delta_0(\overrightarrow G[V']) &\ge \left(\frac{2}{3} + \alpha\right)n,\text{~(ii) }  
   \delta(H[V']) \ge \left(\frac{1}{3} + 5\alpha\right) n, \text{~and}\\ 
   \text{ (iii) }\delta(H[C']) &\ge \left(\frac{1}{3} +
   3\alpha\right)n \text{ for }C'\in \{A',B'\}.\notag
   \end{align}
  In particular, (\ref{V'}iii) implies $\frac{n}{3} < |A'|,|B'| < \frac{2n}{3}$.

  If $|A'| \equiv |B'| \equiv 0 \pmod{3}$ then let $A'' := A'$ and $B'' := B'$.
  If not, since $|A' \cup B'|$ is divisible by $3$, 
  without loss of generality we can assume that $|A'| \equiv 1 \pmod{3}$
  and $|B'| \equiv 2 \pmod 3$.  
  Let $v \in B'$.   Since $|B'| < \frac{2n}{3}$ there exists 
  $u \in N^+(v) \cap A'$. By (\ref{V'}i,iii) $$|N^+(u)\cap (N_H(v)\cap B')|\ge
  \left(\frac{2}{3}+\alpha\right)n+\left(\frac{1}{3} + 3 \alpha\right)n-(n-1)>0.$$
  So there exists $w\in B'-v$ with $T:=uvw=\overrightarrow C_3$. Let
  $A'' := A' \smallsetminus V(T)=A'-v$ and $B'' := B' \smallsetminus V(T)=B'-u-v$.
  In either case, 
  $2(|A'| - 2|B' \smallsetminus B''| + 4) > \frac{2n}{3} > |B''| - 1$, so 
  \begin{align*} 
    \delta_{t}\left(\overrightarrow G[B'']\right) &\ge 
    2\delta_0(\overrightarrow G[V'])-(\|v,A'\|_{G_M}+\|v,A'\|_{H_M})-\|v,B'\smallsetminus B''\|_M\\
    &\ge
    2 \left(\frac{2}{3} +  \alpha \right)n - \left(|A'| + 2 \alpha n\right) - 
    2|B' \smallsetminus B''|  
    &&\text{(by (\ref{V'}i)})\\
   &= 
   \frac{4}{3}\left(n - |A'| - |B' \smallsetminus B''| - 1\right) + 
    \frac{1}{3}\left(|A'| - 2|B' \smallsetminus B''| + 4\right)\\
    &> \frac{3}{2}\left(|B''| - 1\right).
  \end{align*}
  A similarly calculation gives that
  $\delta_{t}(\overrightarrow G[A'']) \ge \frac{3}{2}(|A''| - 1)$.
  Hence, by Theorem~\ref{thm:Wang}, 
  both $\overrightarrow G[A'']$ and $\overrightarrow G[B'']$ have cyclic triangle factors.
\end{proof}

\begin{proof}[Proof of Theorem~\ref{thm:5tri_factor}.]

  This proof uses the probabilistic absorbing method 
  as in \cite{levitt2010avoid}.
  Let $0 < \sigma < \min\{\frac{\varepsilon}{12}, \frac{\sqrt{\alpha}}{16}\}$
  and $\tau := \frac{\sigma^{45}}{4}$ and assume throughout
  that $n$ is sufficiently large.  Let $H := H_M$.

  \begin{defn}
    For any disjoint $X,Y \subseteq V$, we will say that 
    $Y$ \emph{absorbs} $X$ if
    $M[Y]$ and $M[Y \cup X]$ both have $5$-triangle factors.
    For any $Z := (z_1, \dotsc, z_{45}) \in V^{45}$
    let $V(Z) := \{z_1, \dotsc, z_{45}\}$.
    For any $X\in\binom{V}{3}$, 
    call $Z$ an $X$-\emph{sponge} 
    when $|V(Z)| = 45$
    and $V(Z)$ absorbs $X$, 
    and let $\mathcal{A}_{X}$ be the set of $X$-sponges.
    Two sponges $Z,Z'$ are \emph{disjoint} if $V(Z)$ and $V(Z')$ are disjoint.
    For any collection of sponges $\mathcal{A}$ let
    $V(\mathcal{A}) := \bigcup_{Z \in \mathcal{A}}V(Z)$.
  \end{defn}

  \begin{defn}
    For $k > 0$ the tuple $(z_1, \dotsc, z_{3k-1}) \in V^{3k-1}$ 
    is a \emph{$k$-chain} if
    \begin{enumerate}[label=(\alph*),nolistsep, ref={\theclaim~(\alph*)}]
      \item $z_1, \dotsc, z_{3k-1}$ are distinct vertices,
      \item $z_{3i-2}z_{3i-1}$ is a heavy edge for $1 \le i \le k$, and
      \item $z_{3i} \in Q_3(z_{3i-2}z_{3i-1}) \cap Q_3(z_{3i+1}z_{3i+2})$ 
	for $1 \le i \le k - 1$.
    \end{enumerate}
    For $u, v \in V$ if $u \in  Q_3(z_1z_2)$ and 
    $v \in Q_3(z_{3i-2}z_{3i-1})$ for some $1 \le i \le k$
    and $u, v \notin \{z_1, \dotsc, z_{3k-1}\}$
    then we say that the $k$-chain \emph{joins $u$ and $v$} 
    (see Figure~\ref{fig:5chain}).
    For $k > 0$, if there are at least $(\sigma n)^{3k-1}$ $k$-chains
    that join $u$ and $v$ we say that
    \emph{$u$ is $k$-joined with $v$}.
  \end{defn}

\begin{figure}[htb]
  \centering 

  \global\long\def\svgwidth{400pt}
  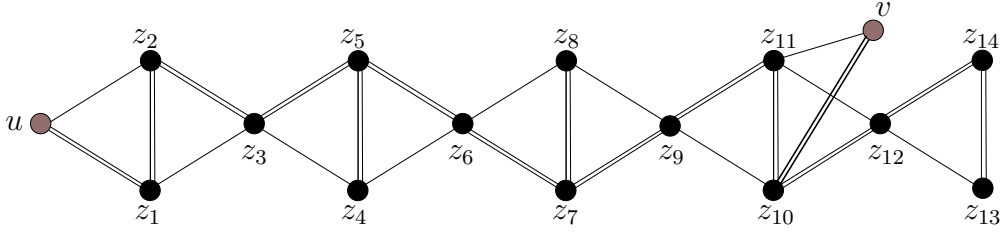 \caption{The $5$-chain $(z_1, \dotsc, z_{14})$ joins $u$ and $v$.}

  \label{fig:5chain} 
\end{figure}
  Note that for $1 \le i < k \le 5$ 
  if $u$ is $i$-joined with $v$ then 
  $u$ is $k$-joined with $v$.
  Indeed, using \eqref{eq:q_3} and \eqref{eq:f_u}, 
  we can extend any $i$-chain that joins $u$ and $v$
  by iteratively picking a vertex $z_{3j} \in Q_3(z_{3j - 2}, z_{3j - 1})$
  and then a heavy edge $z_{3j+1}z_{3j+1} \in F(z_{3j})$
  that avoids the vertices $\{u, v, z_1, \dotsc, z_{3j - 1}\}$
  for $j$ from $i+1$ to $k$ in at least $(\sigma n)^{3j}$ ways.
  For any $u \in V$ define
  \begin{equation*}
    L_k(u) := 
    \{v \in V : \text{$v$ is $k$-joined with $u$ for some $1 \le i \le k$}\}.
  \end{equation*}
  Note that, by the previous comment,
  $L_1(u) \subseteq \dotsm \subseteq L_5(u)$.

  Let $\{x_1, x_2, x_3\} := X \in \binom{V}{3}$,
  $Y := (z_1, \dotsc, z_{45}) \in V^{45}$, and define 
   $m(i) := 15(i-1)$.
  It is not hard to see that $Y \in \mathcal{A}_X$
  if $Y$ satisfies the following for $i \in [3]$:
  \begin{itemize}[nolistsep]
    \item
      the vertices $z_1, \dotsc, z_{45}$ are distinct,
    \item
      $M[\{z_{m(1) + 1}, z_{m(2) + 1}, z_{m(3) + 1}\}]$ is a $5$-triangle, and
    \item  
      $(z_{m(i) + 2}, \dotsc, z_{m(i) + 15})$ is a $5$-chain that
      joins $z_{m(i) + 1}$ and $x_i$.
  \end{itemize}

  Our plan is to show (i) there is a small set $\mathcal{A}$ of disjoint
  sponges such that for all $3$-sets $X\in\binom{V}{3}$ there exists
  an $X$-sponge $Y\in\mathcal{A}$, and (ii) there exists a $3$-set
  $X\subseteq V\smallsetminus V(\mathcal{A})$ such that $M-(X\cup
  V(\mathcal{A}))$
  has a $5$-triangle factor. Since there exists an
  $X$-sponge in $\mathcal{A}$, this will imply that $M$ has a $5$- 
  triangle factor.

  To prove (i) we first show that every $3$-set is absorbed by a positive
  fraction of all $45$-tuples, and then use Chernoff's inequality to find
  \emph{$\mathcal{A}$}. The following claim is our main tool.

  \setcounter{claim}{0}
  
  \begin{claim}
    $L_5(x) = V$ for every vertex $x \in V$.
    \label{lemma:connecting_edges_5_triangles}
  \end{claim}
  \begin{proof}
    We will first show that, for every $u \in V$,
    \begin{equation*}
      |L_1(u)| \ge 
      \left(\frac{1}{3} + \frac{\varepsilon}{3}\right)n \text{ and }
      u \in L_1(u).
    \end{equation*}
    By \eqref{eq:f_u}, $|F(u)| \ge (\alpha n)^2$, so $u \in L_1(u)$.
    Let 
    $t := \sum_{e \in F(u)} |Q_3(e)|$.
    By \eqref{eq:q_3}, 
    $t \ge |F(u)|\left(\frac{1}{3} + \varepsilon\right)n$.
    If $v \notin L_1(u)$ then
    there are less than 
    $(\sigma n)^2 < \varepsilon (\alpha n)^2 \le \varepsilon |F(u)|$ edges $e \in F(u)$
    for which $v \in Q_3(e)$.
    Therefore,
    \begin{equation*}
      |F(u)|\left(\frac{1}{3} + \varepsilon \right)n \le 
      t < 
      |F(u)||L_1(u)| + \varepsilon|F(u)||\overline{L_1(u)}|
      \le \varepsilon|F(u)| n + 
      \left(1 - \varepsilon\right)|F(u)||L_1(u)|,
    \end{equation*}
    and $|L_1(u)| > 
    \frac{n}{3} \cdot (1 - \varepsilon)^{-1}
    > \left(\frac{1}{3} + \frac{\varepsilon}{3} \right)n$. 

    Note that for any $u, v \in V$ 
    if $|L_i(u) \cap L_j(v)| \ge 2 \sigma n$ and
    $1 \le i,j \le 2$ then $v \in L_{i+j}(u)$.
    Indeed, we can pick $w \in L_i(u) \cap L_j(v)$ in one
    of $2 \sigma n$ ways and 
    we can then pick 
    an $i$-chain $(u_1, \dotsc, u_{3i - 1})$ that joins $u$ and $w$
    and a $j$-chain $(v_1, \dotsc, v_{3i - 1})$ that joins $v$ and $w$
    so that $u, u_1, \dotsc, u_i, w, v_j, \dotsc, v_1$ and $v$
    are all distinct in $\frac{1}{2} (\sigma n)^{3(i + j)- 2}$ ways.
    Since $(u_1, \dotsc, u_i, w, v_j, \dotsc, v_1)$ is a $(i+j)$-chain
    that joins $u$ and $v$ and there are
    $(\sigma n)^{3(i + j)-1}$ such $2$-chains, $v \in L_2(u)$.

    Let $x \in V$ and suppose, by way of contradiction,
    that there exists $y \in V$ such that $y \notin L_5(x)$.
    If there exists $z \notin L_2(x) \cup L_2(y)$,
    from the preceding argument,
    we have $|L_1(u) \cap L_1(v)| < 2\sigma n$ for any 
    distinct $u, v \in \{x, y, z\}$.
    But this is a contradiction, because
    $3\left(\frac{1}{3} + \frac{\varepsilon}{3}\right)n - 3(2 \sigma) n > n$.
    Therefore, if we let $X := L_2(x)$ and $Y := L_2(y) \smallsetminus L_2(x)$,
    $\{X, Y\}$ is a partition of $V$.
    We have that 
    $|X| \ge |L_1(x)| \ge \left( \frac{1}{3} + \frac{\varepsilon}{3}\right)n$
    and, since $y \notin L_4(x)$, 
    $|L_2(y) \cap L_2(x)| < 2\sigma n$
    so 
    $|Y| \ge |L_1(y)| - 2 \sigma n \ge
    \left( \frac{1}{3} + \frac{\varepsilon}{6}\right)n$

    Call a $4$-tuple $(v_1, v_2, v_3, v_4)$ \emph{connecting}
    if $v_1 \in X$ and $v_4 \in Y$, $v_2v_3 \in E_H$
    and $v_1,v_4 \in Q_3(v_2v_3)$.
    Since $M$ is not $\alpha$-splittable, $|E_H(X, Y)| \ge \alpha n^2$.
    Pick some $e := x'y' \in E_H(X, Y)$ where $x' \in X$ and $y' \in Y$.
    We will show that there are at 
    least $(\sigma n)^2$
    connecting $4$-tuples which contain $x'$ and $y'$.
    Since $M[{v_1, v_2, v_3, v_4}]$ can contain at most 
    $4$ edges from $E_H(X, Y)$, this will imply that
    there are at least
    $\frac{1}{4} \cdot \alpha n^2 \cdot (\sigma n)^2 \ge 4(\sigma n)^4$
    connecting $4$-tuples and
    this will prove that $y \in L_5(x)$, a contradiction.
    Indeed, 
    select a connecting $4$-tuple $(v_1, v_2, v_3, v_4)$ in 
    $4 (\sigma n)^4$ ways.
    Since $v_1$ is $2$-joined with $x$ there
    are at least $\frac{1}{2}(\sigma n)^{5}$
    $2$-chains that join $x$ and $v_1$
    and avoid $\{v_1, v_2, v_3, v_4\}$.
    Similarly,
    there are $\frac{1}{2}(\sigma n)^{5}$ $2$-chains
    that join $v_4$ and $y$ and avoid all previously selected vertices. 
    Therefore, there are
    at least $(\sigma n)^{14}$ $5$-chains that join $x$ and $y$.
    So, by way of contradiction, assume there
    are less than $(\sigma n)^2$ connecting $4$-tuples containing $e$. 

    Suppose $|Q_4(e)| \ge \sigma n$ and pick $z \in Q_4(e)$
    and let $T := \{x', y', z\}$.
    Note that $M[T]$ is a $6$-triangle and that
    \begin{equation*}
      3 \left(\frac{4}{3} + \varepsilon\right)n \le \|T, V\|
      \le 6|Q_5(T)| + 4|\overline{Q_5(T)}| = 
      2|Q_5(T)| + 4n
    \end{equation*}
    so $|Q_5(T)| \ge \frac{3}{2} \varepsilon n$.
    Pick $w \in Q_5(T)$.
    Note that there are
    at least  $\sigma n \cdot \frac{3}{2} \varepsilon n \ge (\sigma n)^2$ 
    choices for the pair $(z, w)$
    and that if $w \in X$
    then $(w, x', z, y')$ is a connecting $4$-tuple
    and if $w \in Y$ then $(x', z, y', w)$ is a connecting $4$-tuple.
    Therefore, we can assume $|Q_4(e)| < \sigma n$ which, 
    by \eqref{eq:q_4_plus_q_3},
    implies that $|Q_3(e)| \ge \left(\frac{2}{3} + \varepsilon\right)n$.

    For any $v_1 \in Q_3(e) \cap X$ and $v_4 \in Q_3(e) \cap Y$,
    $(v_1, x', y', v_4)$ is a connecting $4$-tuple.
    Therefore, we cannot have
    $|Q_3(e) \cap X| \ge \sigma n$ and $|Q_3(e) \cap Y| \ge \sigma n$.
    So suppose $|Q_3(e) \cap X| < \sigma n$.
    Then $|Y| \ge |Q_3(e) \cap Y| > \frac{2n}{3}$ 
    which contradicts the fact that $|X| > \frac{n}{3}$.
    Since a similar argument holds when $|Q_3(e) \cap Y| < \sigma n$,
    the proof is complete.
  \end{proof}


  \begin{claim}
    For every $X\in\binom{V}{3}$, $|\mathcal{A}_{X}|\ge 4\tau n^{45}$.
    \label{lemma:absorbing_tuples}
  \end{claim}
  \begin{proof}
    Recall that $m(i) := 15(i - 1)$ and
    let $\{x_1, x_2, x_3\} := X$.
    Pick $v_{m(1) + 1}v_{m(2) + 1} := e$ 
    from one of the at least $\frac{1}{3}n^2$ edges in $H - X$.
    By \eqref{eq:q_3}, we can pick $v_{m(3) + 1}$ 
    from one of the more than $\frac{1}{3}n$ 
    vertices in $Q_3(e) \smallsetminus X$. 
    We have that $M[\{v_{m(1) + 1},v_{m(2) + 1},v_{m(3) + 1}]$
    is a $5$-triangle.
    For $i \in [3]$, pick
    a $5$-chain $(v_{m(i) + 2}, \dotsc, v_{m(i)+15})$ that joins
    $v_{m(i) + 1}$ and $x_i$ in one of 
    $\frac{1}{2}(\sigma n)^{14}$ ways.
    Note that $(v_1, \dotsc, v_{45}) \in \mathcal{A}_X$ and
    there are at least
    $\frac{1}{72} \sigma^{42} n^{45} \ge 4 \tau n^{45}$ such tuples.
  \end{proof}

  The next part of the proof is probabilistic. The tools we require
  are the union bound, the linearity of expectation, Markov's inequality
  and the following theorem of Chernoff:
  \begin{thm}
    Let $X$ be a random variable with binomial distribution. Then the
    following hold for any $t\ge0$: 
    \begin{enumerate}[label=(\alph*),nolistsep, ref={\theclaim~(\alph*)}]
      \item $\Pr[X\ge\E[X]+t]\le\exp\left(-\frac{t^{2}}{2(\E[X]+t/3)}\right)$; 
      \item $\Pr[X\leq\E[X]-t]\le\exp\left(-\frac{t^{2}}{2\E[X]}\right)$. 
    \end{enumerate}
    \label{chernoff} 
  \end{thm}

  The next lemma completes step (i) of the proof.
  \begin{lem}
    There exists a set $\mathcal{A}$
    of disjoint sponges such that 
    $|\mathcal{A}|\leq \frac{\varepsilon}{90} n$
    and for every $3$-set $X$ there exists an $X$-sponge in $\mathcal{A}.$
    \label{lemma:existence_of_the_absorbing_structure}
  \end{lem}
  \begin{proof}
    Let $\mathcal{F}$ be a set of $45$-tuples chosen by picking each
    $45$-tuple $Y\in V^{45}$ randomly and independently with probability
    $\rho n^{-44}$ where 
    $\rho := \frac{\tau}{4 \cdot 10^3}$. 
    Let $\varepsilon' := \varepsilon/90$ and note that $\rho < \varepsilon'$.
    Clearly $\E[|\mathcal{F}|]=\rho n$, and, by Claim
    \ref{lemma:absorbing_tuples}, 
    $\E[|\mathcal{F}\cap\mathcal{A}_{X}|]\ge 4\tau\rho n$ 
    for each $X\in\binom{V}{3}$. 
    Let $\mathcal{O}$ be the set of pairs
    of overlapping tuples in $V^{45}$, that is 
    \begin{equation*}
      \mathcal{O}=\left\{
	\{\left(x_{1},\dotsc,x_{45}\right),\left(y_{1},\dotsc,y_{45}\right)\}\in\binom{V^{45}}{2}:x_{i}=y_{j}\text{
	for some }i,j\in\{1,\dotsc,45\}\right\} ;
      \end{equation*}
      note that 
      $|\mathcal{O}|\le\frac{1}{2}(n\cdot45\cdot45\cdot n^{88}) < 2000 n^{89}$.
      For any $P\in\mathcal{O}$, $\Pr[P\subseteq\mathcal{F}]=(\rho n^{-44})^{2},$
      so if 
      $\mathcal{O}_{\mathcal{F}} :=
      \{P\in\mathcal{O}:P\subseteq\mathcal{F}\}$
      then, by the linearity of expectation,  
      \begin{equation*}
	\E[|\mathcal{O}_{\mathcal{F}}|]\le
	\sum_{P\in\mathcal{O}}\Pr[P\subseteq\mathcal{F}] <
	2000 n^{89}\cdot\rho^{2}n^{-88} = \frac{\tau \rho}{2} n.
      \end{equation*}

      By Theorem~\ref{chernoff}, 
      \begin{equation*}
	\Pr[|\mathcal{F}|\ge\varepsilon' n]=\Pr[|\mathcal{F}|\ge\E[|\mathcal{F}|]+(\varepsilon'-\rho)n]\le\exp\left(-\frac{3(\varepsilon'-\rho)^{2}n}{2(2\rho+\varepsilon')}\right)
      \end{equation*}
      and, for every $X\in\binom{V}{3}$, 
      \begin{equation*}
	\Pr[|\mathcal{A}_{X}\cap\mathcal{F}|\le 3 \tau \rho n]=
	\Pr[|\mathcal{A}_{X}\cap\mathcal{F}|\le\E[X] - \tau \rho
      n]\le\exp\left(-\frac{(\tau \rho )^2n}{8 \tau\rho}\right)
    \end{equation*}
    Since by Markov's inequality we have that 
    $\Pr[|\mathcal{O}_{\mathcal{F}}|\ge \tau \rho n]\le\frac{1}{2}$,
    if $n$ is large enough
    \begin{equation*}
      \Pr[|\mathcal{F}|\ge\varepsilon' n]+
      \sum_{X\in\binom{V}{3}}{\Pr[|\mathcal{A}_{X}\cap\mathcal{F}|\le 3 \tau\rho
      n]}+\Pr[|\mathcal{O}_{\mathcal{F}}|\ge \tau \rho n]<1.
    \end{equation*}
    Therefore, by the union bound, we can fix $\mathcal{F} \subseteq V^{45}$
    so that, $|\mathcal{F}|<\varepsilon' n$, $|\mathcal{A}_{X}\cap\mathcal{F}|> 3 \tau \rho n$
    for every $X\in\binom{V}{3}$, and $|\mathcal{O}_{\mathcal{F}}|<\tau \rho n$.

    Note that 
    $|\bigcup\mathcal{O}_{\mathcal{F}}| \le 
    2|\mathcal{O}_{\mathcal{F}}| \le 2 \tau \rho n$ and let 
    \begin{equation*}
      \mathcal{A}:=\left\{ 
	T\in\mathcal{F}\smallsetminus\left(\bigcup\mathcal{O}_{\mathcal{F}}\right)
	:T\in\mathcal{A}_{X}\text{ for some }
	X\in\binom{V}{3}\right\}.
    \end{equation*}
      For every $Z\in\mathcal{A}$ 
      there exists $X\in\binom{V}{3}$, 
      such that $Z \in \mathcal{A}_X$,
      so $\mathcal{A}$ is a collection of sponges.
      Also, note that for distinct 
      $Z,Z'\in\mathcal{A}$, $\{Z,Z'\}\notin\mathcal{O}$,
      that is, the sponges in $\mathcal{A}$ are disjoint. Furthermore, 
      for any $X \in \binom{V}{3}$,
      $|\mathcal{A}_{X}\cap\mathcal{A}|\ge
      |\mathcal{A}_{X}\cap\mathcal{F}|-|\bigcup \mathcal{O}_{\mathcal{F}}|\ge
      \left \lceil \tau \rho n \right \rceil \ge1$.
    \end{proof}

    Let $\mathcal{A}$ be the set of $45$-tuples guaranteed by
    Lemma~\ref{lemma:existence_of_the_absorbing_structure}
    and let $A:=V(\mathcal{A})$.
    Let $M' = M - A$. 
    By Theorem~\ref{th:main},
    because $|A| \le 45 \varepsilon'n \le \frac{\varepsilon}{2} n$,
    there is a $5$-triangle packing of $M' - X$ 
    where $X$ is a $3$-set of $V \smallsetminus A$.
    There exists
    $Z \in \mathcal{A}_{X}\cap\mathcal{A}$.
    By the definition of an $X$-sponge there is a $5$-triangle
    factor of $M[X \cup V(Z)]$ and, 
    since every tuple in $\mathcal{A}$ is a sponge,
    there is a $5$-triangle factor
    of $M[A \smallsetminus V(Z)]$. This completes the proof.
  \end{proof}

  \hbadness 10000\relax

  \end{document}